\documentclass[10pt,leqno]{article}
\usepackage[T1]{fontenc}
\usepackage{lmodern} 
\usepackage[utf8]{inputenc}
\usepackage[a4paper]{geometry} 
\usepackage{indentfirst}
\usepackage{amsfonts,amsmath,bm,amssymb,amsthm,latexsym,mathtools,url,textcomp}
\usepackage[all]{xy}
\usepackage[normal]{caption}

\newcommand{\arXiv}[1]{\href{http://arxiv.org/abs/#1}{\texttt{arXiv:#1}}}
\newcommand{\Hal}[1]{\href{https://hal.archives-ouvertes.fr/#1}{\texttt{HAL:#1}}}
\usepackage{color}
\usepackage[colorlinks=true,citecolor=cyan,linkcolor=magenta,urlcolor=blue]{hyperref}

\newlength{\wideitemsep}		
	\let\olditem\item
\renewcommand{\item}{\setlength{\itemsep}{\wideitemsep}\olditem}

\def\rd{\mathrm{d}}
\def\ri{\mathrm{i}}
\def\re{\mathrm{e}}

\newcommand{\C}{{\mathbb C}}

\theoremstyle{plain}
\newtheorem{thm}{Theorem}[section]
\newtheorem{lem}[thm]{Lemma}

\theoremstyle{definition}

\newtheorem{rem}{Remark}

\newcommand{\bi}{\begin{itemize}}
\newcommand{\ei}{\end{itemize}}
\newcommand{\bd}{\begin{description}}
\newcommand{\ed}{\end{description}}
\newcommand{\be}{\begin{enumerate}}
\newcommand{\ee}{\end{enumerate}}

\def\bc{\begin{center}}
\def\ec{\end{center}}

\def\bf{\textbf}

\newcommand{\tand}{\text{and}}

\newcommand{\lsum}{\sum\limits}
\newcommand{\lint}{\int\limits}

\def\z{\zeta}
\def\eps{\epsilon}

\def\no{\noindent}
\def\l{\left}
\def\r{\right}
\def\bl{\bigl}
\def\br{\bigr}

\def\m{\medskip}
\def\s{\smallskip}

\begin{document}

\title{Integral representations and asymptotic behaviours 
of Mittag-Leffler type functions of two variables}
\author{Christian Lavault\thanks{LIPN, CNRS UMR 7030. \emph{E-mail:}\
\href{mailto:lavault@lipn.univ-paris13.fr}{lavault@lipn.univ-paris13.fr}}
}

\date{\small \empty}
\maketitle

\begin{abstract}
The paper explores various special functions which generalize the two-parametric Mittag-Leffler type function of two variables. Integral representations for these functions in different domains of variation of arguments for certain values of the parameters are obtained. The asymptotic expansions formulas and asymptotic properties of such functions are also established for large values of the variables. This provides statements of theorems for these formulas and their corresponding properties.

\s \no {\bf Keywords:} Generalized two-parametric Mittag-Leffler type functions of two variables; Integral representations; Special functions; Hankel's integral contour; Asymptotic expansion formulas.

\s \no 2010 Mathematics Subject Classification: 33E12, 33C70, 11S23, 32A26, 33C50, 41A60.
\end{abstract}

\section{Definition and notation} \label{def}
Let the power series $E_{\alpha,\beta}(z) := \lsum_{n\ge 0} \frac{z^n}{\Gamma(\alpha n + \beta)}\ %
\ (\alpha, \beta\in \C;\ \Re(\alpha) > 0)$ define the classical two-parametric Mittag-Leffler (M-L for short) function (see e.g, \cite[\color{cyan}1953]{Agar53,HumAgr53}, \cite[\color{cyan}1905]{Wiman05}). In the case when $\alpha$ and $\beta$ are real positive, the series converges for all values of $z\in \C$, so $E_{\alpha,\beta}(z)$ is an entire functions of $z\in \C$~\cite[\color{cyan}Lect.~1]{Levin96} of order $\rho = 1/\Re(\alpha)$ and type $\sigma = 1$ (see~\cite[\color{cyan}\S1.1]{Lavault17}). From here on, this latter two-parametric M-L function of $z\in \C$ is denoted for simplicity by $E_\alpha(z; \beta)$, as defined in~\cite{Djrba60,Djrba66}.

The two-parametric M-L function of $z\in \C$ extends to the generalized M-L type function $E_{\alpha,\beta}(x,y;\mu)$ of two variables $x, y\in \C$. The latter is an entire function defined by the double power series~\cite[\color{cyan}Eq.~12]{OgorYasha10}
\begin{equation} \label{def_ml}
E_{\alpha,\beta}(x,y;\mu) := \lsum_{n,m\ge 0} \frac{x^n y^m}{\Gamma(n\alpha + m\beta + \mu)}\ \qquad (\alpha,\, \beta,\, \mu\in \C,\ \Re(\alpha),\, \Re(\beta) > 0),
\end{equation}
where the arbitrary parameter $\mu$ takes in general a complex value. 

Following~\cite{Djrba60,Djrba66} (see also e.g. \cite{GoKiMaRo14,HaMaSa11,Temme96}, \cite[\color{cyan}\S1.2 \& App.~A, C \& D]{Lavault17}, and references therein), the Hankel's integral contour is denoted by $\gamma(\eps;\theta) := \bl\{0 < \theta\le \pi,\ \eps > 0\br\}$ oriented by non-decreasing $\arg \z$. It consists in the following parts:
\be
\item the two rays $S_{\theta} =  \bl\{\arg \z = \theta,\ |\z|\ge \eps\br\}$ and $S_{-\theta} = %
\bl\{\arg \z = -\theta,\ |\z|\ge \eps\br\}$;
\item the circular arc $C_\theta(0;\eps) =  \bl\{|\z| = \eps,\ -\theta\le \arg \z\le \theta\br\}$.
\ee
If $0 < \theta < \pi$, then the Hankel contour $\gamma(\eps;\theta)$ divides the complex $\z$-plane into two unbounded regions, namely $\Omega^{(-)}(\eps;\theta)$ to the left of $\gamma(\eps ;\theta)$ by orientation and $\Omega^{(+)}(\eps;\theta)$ to the right of it. If $\theta = \pi$, then the contour consists of the circle $|\z| = \eps$ and the twice passable ray $-\infty < \z\le -\eps$.

\section{Integral representations} \label{rep}
This section provides a few lemmas, which show various integral representations of the generalized M-L type function~\eqref{def_ml} corresponding to different domains of variation  of its arguments.

\begin{lem} \label{lem1}
Let $0 < \alpha, \beta < 2$, $\alpha\beta < 2$. Let $\mu$ be any complex number and let $\theta$ meet the condition 
\begin{equation} \label{cond1}
\pi \alpha\beta/2 < \theta\le \min\bl(\pi, \pi\alpha \beta\br).
\end{equation}

If $x\in \Omega^{(-)}(\eps_\alpha;\theta_\alpha)$\ and $y\in \Omega^{(-)}(\eps_\beta;\theta_\beta)$, where $\eps_\alpha := \eps^{1/\beta}$, $\eps_\beta := \eps^{1/\alpha}$, $\theta_\alpha := \theta/\beta$\ and $\theta_\beta := \theta/\alpha$, then the integral representation based on Hankel's contour integral holds
\begin{equation} \label{mlrep1}
E_{\alpha,\beta}(x,y;\mu) = \frac{1}{2\pi\ri} \frac{1}{\alpha\beta} \lint_{\gamma(\eps;\theta)} %
\frac{\re^{\z^{1/(\alpha \beta)}} \z^{\frac{\alpha + \beta - \mu}{\alpha \beta} - 1}} %
{(\z^{1/\alpha} - y)(\z^{1/\beta} - x)}\, \rd \z.
\end{equation}
\end{lem}
\begin{proof}
First, let $|x| < \eps_\alpha$. Taking into account the fact that $\eps_\alpha = \eps^{1/\beta} = %
\l(\eps_\beta^\alpha\r)^{1/\beta} = \eps_\beta^{\alpha/\beta}$ yields
\[\sup_{\z\in \gamma(\eps_\beta;\theta_\beta)} \bl|x\z^{-\alpha/\beta}\br| < 1.\]

By the statement of the definition in~\eqref{def_ml}, the expansion of $E_{\alpha,\beta}(x,y;\mu)$ may be rewritten as follows in terms of the corresponding two-parametric M-L function $E_\beta(y; \alpha n + m)$ of one variable,
\begin{flalign} \label{oneparml}
E_{\alpha,\beta}(x,y;\mu) &= \lsum_{n=0}^\infty \lsum_{m=0}^\infty \frac{x^n y^m} %
{\Gamma(\alpha n + \beta m + \mu)} \nonumber\\
&= \lsum_{n=0}^\infty x^n \lsum_{m=0}^\infty \frac{y^m}{\Gamma(\beta m + (\alpha n + \mu))} %
= \lsum_{n=0}^\infty x^n E_\beta(y; \alpha n + m).
\end{flalign}
Under the assumptions of Lemma~\ref{lem1}, it is possible to use the known integral representation of $E_\beta(y; \alpha n + m)$ (see e.g. \cite[\color{cyan}Eq.~(2.2)]{Djrba66}). Taking the above $\eps_\beta$ and $\theta_\beta$ as the parameters defining the Hankel contour, which is admissible according to inequalities~\eqref{cond1} provided that $\theta_\beta = \theta/\alpha$, gives for $y\in \Omega^{(-)}(\eps_\beta;\theta_\beta)$,
\begin{flalign} \label{intrep1}
E_{\alpha,\beta}(x,y;\mu) &= \lsum_{n=0}^\infty x^n E_\beta(y; \alpha n + m) %
= \lsum_{n=0}^\infty x^n\, \frac{1}{2\pi\ri} \frac{1}{\beta} \lint_{\gamma(\eps_\beta;\theta_\beta)} %
\frac{\re^{\z^{1/\beta}} \z^{\frac{1 - \alpha n - \mu}{\beta}}} {\z-y}\, \rd \z\nonumber\\
&= \frac{1}{2\pi\ri} \frac{1}{\beta} \lint_{\gamma(\eps_\beta;\theta_\beta)} %
\frac{\re^{\z^{1/\beta}} \z^{\frac{1-\mu}{\beta}}} {\z-y}\, \l(\lsum_{n=0}^\infty %
\l(x\z^{\alpha/\beta}\r)^{n}\r)\, \rd \z \nonumber\\
&= \frac{1}{2\pi\ri} \frac{1}{\beta} \lint_{\gamma(\eps_\beta;\theta_\beta)} %
\frac{\re^{\z^{1/\beta}} \z^{\frac{1 + \alpha - \mu}{\beta}}} %
{(\z - y)(\z^{\alpha/\beta} - x)}\, \rd \z.
\end{flalign}
Now, rewriting the above integral representation~\eqref{intrep1} along the suitable integral contour 
$\gamma(\eps;\theta)$, we get Eq.~\eqref{mlrep1},
\begin{flalign*}
E_{\alpha,\beta}(x,y;\mu) &= \frac{1}{2\pi\ri} \frac{1}{\beta} \lint_{\gamma(\eps;\theta)} %
\frac{ \re^{\l(\xi^{1/\alpha}\r)^{1/\beta}} \l(\xi^{1/\alpha}\r)^{\frac{1 + \alpha - \mu}{\beta} } } %
{(\xi^{1/\alpha} - y)(\xi^{1/\beta} - x)}\, \frac{1}{\alpha}\, %
\xi^{\frac{1 - \alpha}{\alpha}}\, \rd \xi\nonumber\\
&= \frac{1}{2\pi\ri} \frac{1}{\alpha\beta} \lint_{\gamma(\eps;\theta)} %
\frac{\re^{\xi^{1/(\alpha \beta)}} \xi^{\frac{\alpha + \beta - \mu}{\alpha \beta} - 1}} %
{(\xi^{1/\alpha} - y)(\xi^{1/\beta} - x)}\, \rd \xi.
\end{flalign*}
The above resulting integral is absolutely convergent and it is an analytic function of $x$ and $y$ for 
$x\in \Omega^{(-)}(\eps_\alpha;\theta_\alpha)$, $y\in \Omega^{(-)}(\eps_\beta;\theta_\beta)$.

The open disk $D = \{|x| < \eps_\alpha\}$ is contained into the complex region $\Omega^{(-)}(\eps_\alpha;\theta_\alpha)$ for all values of $\theta_\alpha$ in the open interval $\bl]\pi\alpha/2, \min(\pi,\pi \alpha)\br[$. Therefore, from the principle of analytic continuation Eq.~\eqref{mlrep1} is valid everywhere within the complex region $\Omega^{(-)}(\eps_\alpha;\theta_\alpha)$ and the lemma follows.
\end{proof}

\begin{lem} \label{lem2}
Let $0 < \alpha, \beta < 2$, $\alpha\beta < 2$ Let $\mu$ be any complex number and let $\theta$ verify inequalities~\eqref{cond1}: $\pi \alpha\beta/2 < \theta\le \min\bl(\pi, \pi \alpha\beta\br)$.

If $x\in \Omega^{(-)}(\eps_\alpha;\theta_\alpha)$\ and $y\in \Omega^{(+)}(\eps_\beta;\theta_\beta)$, where $\eps_\alpha := \eps^{1/\beta}$, $\eps_\beta := \eps^{1/\alpha}$, $\theta_\alpha := \theta/\beta$\ and $\theta_\beta := \theta/\alpha$, then the integral representation holds
\begin{equation} \label{mlrep2}
E_{\alpha,\beta}(x,y;\mu) = \frac{1}{\beta}\, \frac{ \re^{y^{1/\beta}} %
y^{\frac{1 + \alpha - \mu}{\beta}} } {y^{\alpha/\beta} - x} + \frac{1}{2\pi\ri} \frac{1}{\alpha\beta} %
\lint_{\gamma(\eps;\theta)} \frac{ \re^{\z^{1/(\alpha \beta)} } \z^{\frac{\alpha + \beta - \mu} %
{\alpha \beta} - 1} } {(\z^{1/\alpha} - y)(\z^{1/\beta} - x)}\, \rd \z.
\end{equation}
\end{lem}
\begin{proof}
By assumption, the point $y$ is located to the right of the Hankel contour $\gamma(\eps;\theta)$, that is $y\in \Omega^{(+)}(\eps_\beta;\theta_\beta)$. Then, for any $\eps_{\beta_1} > y$, $y\in \Omega^{(-)}(\eps_{\beta_1};\theta_\beta)$\ and $x\in \Omega^{(-)}(\eps_{\alpha_1};\theta_\alpha)$. Thus, by~\eqref{intrep1} we get the integral representation
\begin{equation} \label{intrep2}
E_{\alpha,\beta}(x,y;\mu) = \frac{1}{2\pi\ri} \frac{1}{\beta} \lint_{\gamma(\eps_{\beta_1} ; %
\theta_\beta)} \frac{ \re^{\z^{1/\beta} } \z^{\frac{\alpha + \beta - \mu}{\beta}} } %
{(\z - y)(\z^{\alpha/\beta} - x)}\, \rd \z.
\end{equation}
On the other hand, if $\eps_\beta < y < \eps_{\beta_1}$, $|\arg y| < \theta_\beta$ and then, by Cauchy theorem,
\begin{equation} \label{intrep3}
E_{\alpha,\beta}(x,y;\mu) = \frac{1}{2\pi\ri} \frac{1}{\beta} \lint_{\gamma(\eps_{\beta_1} ; %
\theta_\beta) - \gamma(\eps_\beta;\theta_\beta)} \frac{ \re^{\z^{1/\beta}} %
\z^{\frac{1 + \alpha - \mu}{\beta}} } {(\z^{1/\alpha} - y)(\z^{1/\beta} - x)}\, \rd \z %
= \frac{1}{\beta}\, \frac{ \re^{y^{1/\beta}} y^{\frac{1 + \alpha - \mu}{\beta}} }{y^{\alpha/\beta} - x}.
\end{equation}
From Eqs.~\eqref{intrep2} and~\eqref{intrep3} we obtain the representation~\eqref{mlrep2}, and Lemma~\ref{lem2} follows.
\end{proof}

\begin{rem} \label{rem1}
For $x\in \Omega^{(+)}(\eps_\alpha;\theta_\alpha)$ and $y\in \Omega^{(-)}(\eps_\beta;\theta_\beta)$\ $E_{\alpha,\beta}(x,y;\mu)$, the integral representation is shown in a similar manner to be
\begin{equation} \label{mlrep3}
E_{\alpha,\beta}(x,y;\mu) = \frac{1}{\alpha}\, \frac{ \re^{x^{1/\alpha}} %
x^{\frac{1 + \beta - \mu}{\alpha}} } {x^{\beta/\alpha} - y} + \frac{1}{2\pi\ri} \frac{1}{\alpha\beta} %
\lint_{\gamma(\eps;\theta)} \frac{ \re^{\z^{1/(\alpha \beta)} } \z^{\frac{\alpha + \beta - \mu} %
{\alpha \beta} - 1} } {(\z^{1/\alpha} - y)(\z^{1/\beta} - x)}\, \rd \z.
\end{equation}
\end{rem}

\begin{lem} \label{lem3}
Let $0 < \alpha, \beta < 2$, $\alpha\beta < 2$. Let $\mu$ be any complex number and let $\theta$ verify inequalities~\eqref{cond1}. 
If $x\in \Omega^{(+)}(\eps_\alpha;\theta_\alpha)$\ and $y\in \Omega^{(+)}(\eps_\beta;\theta_\beta)$, where $\eps_\alpha := \eps^{1/\beta}$, $\eps_\beta := \eps^{1/\alpha}$, $\theta_\alpha := \theta/\beta$\ and $\theta_\beta := \theta/\alpha$, then the integral representation holds
\begin{flalign} \label{mlrep4}
E_{\alpha,\beta}(x,y;\mu) = \frac{1}{\alpha}\, \frac{ \re^{x^{1/\alpha}} %
x^{\frac{1 + \beta - \mu}{\alpha}} }{x^{\beta/\alpha} - y} &+ \frac{1}{\beta}\, \frac{\re^{y^{1/\beta}} %
y^{\frac{1 + \alpha - \mu}{\beta}} } {y^{\alpha/\beta} - x}\nonumber\\
& + \frac{1}{2\pi\ri} \frac{1}{\alpha\beta} \lint_{\gamma(\eps;\theta)} %
\frac{ \re^{\z^{1/(\alpha \beta)} } \z^{\frac{\alpha + \beta - \mu}{\alpha \beta} - 1} } %
{(\z^{1/\alpha} - y)(\z^{1/\beta} - x)}\, \rd \z.
\end{flalign}
\end{lem}
\begin{proof}
By assumption, each of the points $x$ and $y$ lies on the right-hand side of the Hankel contours $\gamma(\eps_\alpha;\theta_\alpha)$ and $\gamma(\eps_\beta;\theta_\beta)$, respectively; that is in the two regions of the complex plane defined by $x\in \Omega^{(+)}(\eps_\alpha;\theta_\alpha)$ and $y\in \Omega^{(+)}(\eps_\beta;\theta_\beta)$ (resp.).
The parameters $\eps_\alpha$ and $\eps_\beta$ correspond to $\eps$. Now choose $\eps_1$ ($\eps_1 > \eps$) such that one of the coordinates is to the right of the contour and the other coordinate to its left (which is always possible provided that $x^\beta\neq y^\alpha$). 

By definition, let $x\in \Omega^{(-)}(\eps_{\alpha_1};\theta_\alpha)$\ and $y\in \Omega^{(+)}(\eps_{\beta_1};\theta_\beta)$ (i.e., $x > y$). Then, by Eq.~\eqref{mlrep2} in Lemma~\ref{lem2}, we have the integral representation
\begin{equation} \label{intrep4}
E_{\alpha,\beta}(x,y;\mu) = \frac{1}{\beta}\, \frac{\re^{y^{1/\beta}} %
y^{\frac{1 + \alpha - \mu}{\beta}}} {y^{\alpha/\beta} - x} + \frac{1}{2\pi\ri} \frac{1}{\alpha \beta} %
\lint_{\gamma(\eps_1;\theta_)} \frac{ \re^{\z^{1/(\alpha\beta)}} %
\z^{\frac{1 + \alpha + \beta - \mu}{\alpha \beta} - 1} }{(\z^{1/\alpha} - y)(\z^{1/\beta} - x)}\, \rd \z.
\end{equation}
\end{proof}
The above integral in~\eqref{intrep4} may be rewritten in the form
\[
\frac{1}{2\pi\ri} \frac{1}{\alpha} \lint_{\gamma(\eps_1;\theta)} \frac{ \re^{\z^{1/\alpha} } %
\z^{\frac{\alpha + \beta - \mu}{\alpha}} } {(\z - y)(\z^{\beta/\alpha} - y)}\, \rd \z.\]
Now, when $\eps_\alpha < x < \eps_{\alpha_1}$, $|\arg x| < \theta_\alpha$ and then, by Cauchy theorem,
\begin{equation} \label{intrep5}
E_{\alpha,\beta}(x,y;\mu) = \frac{1}{2\pi\ri} \frac{1}{\alpha} \lint_{\gamma(\eps_{\alpha_1} ; %
\theta_\alpha) - \gamma(\eps_\alpha;\theta_\alpha)} \frac{ \re^{\z^{1/\alpha}} %
\z^{\frac{1 + \beta - \mu}{\alpha}} } {(\z^{\beta/\alpha} - y)(\z - x)}\, \rd \z = \frac{1}{\alpha}\, %
\frac{ \re^{x^{1/\alpha}} x^{\frac{1 + \beta - \mu}{\alpha}} } {x^{\beta/\alpha} - y}.
\end{equation}
Finally, from Eqs.~\eqref{intrep4} and~\eqref{intrep5} the representation~\eqref{mlrep4} holds true, and the lemma follows.

\begin{lem} \label{lem4}
If $\Re(\mu) > 0$, then the integral representations~\eqref{mlrep1}, \eqref{mlrep2}, \eqref{mlrep3} and \eqref{mlrep4} remain valid for $\alpha = 2$\ or $\beta = 2$. 
\end{lem}
\begin{proof}
Passing  to the limit in the integral representations~\eqref{mlrep1}, \eqref{mlrep2}, \eqref{mlrep3} and \eqref{mlrep4} with respect to the corresponding parameters yields the lemma.
\end{proof}

\section{Asymptotic behaviours} \label{as}
The asymptotic properties of the function $E_{\alpha,\beta}(x,y;\mu)$ for large values of $|x|$ and $|y|$ are of particular interest.

\begin{thm} \label{asthm}
Let $0 < \alpha, \beta < 2$, $\alpha\beta < 2$. Let $\mu$ be any complex number and let $\tau_1$ be any real number satisfying inequalities~\eqref{cond1}
\[
\pi \alpha\beta/2 < \tau_1\le \min\bl(\pi, \pi \alpha\beta\br).\]
Then, for all integer $p\ge 1$, whenever $|x|\to \infty$ and $|y|\to \infty$, the following asymptotic formulas for the function $E_{\alpha,\beta}(x,y;\mu)$ hold from its respective integral representations.
\bi
\item[{\rm 1)}]\ If $|\arg x|\le \tau_1/\beta$\ and $|\arg y|\le \tau_1/\beta$, then
\begin{flalign} \label{as1}
E_{\alpha,\beta}(x,y;\mu) = \frac{1}{\alpha} &\, \frac{ \re^{x^{1/\alpha}} %
x^{\frac{1 + \beta - \mu}{\alpha}} }{x^{\beta/\alpha} - y} + \frac{1}{\beta}\, \frac{\re^{y^{1/\beta}} %
y^{\frac{1 + \alpha - \mu}{\beta}} } {y^{\alpha/\beta} - x}\nonumber\\
&+ \lsum_{n=1}^{p_\beta} \lsum_{m=1}^{p_\alpha} \frac{x^{-n} y^{-m}}{\Gamma(\mu - n\alpha - \m\beta)} %
+ o\l(|xy|^{-1} |x|^{-p_\beta}\r) + o\l(|xy|^{-1} |y|^{-p_\alpha}\r);
\end{flalign}

\item[{\rm 2)}]\ If $|\arg x|\le \tau_1/\beta$\ and $\tau_1/\alpha < |\arg y|\le \pi$, then
\begin{flalign} \label{as2}
E_{\alpha,\beta}(x,y;\mu) &= \frac{1}{\alpha}\, \frac{ \re^{x^{1/\alpha}} %
x^{\frac{1 + \beta - \mu}{\alpha}} }{x^{\beta/\alpha} - y}\nonumber\\
+& \lsum_{n=1}^{p_\beta} \lsum_{m=1}^{p_\alpha} \frac{x^{-n} y^{-m}}{\Gamma(\mu - n\alpha - \m\beta)} %
+ o\l(|xy|^{-1} |x|^{-p_\beta}\r) + o\l(|xy|^{-1} |y|^{-p_\alpha}\r);
\end{flalign}

\item[{\rm 3)}]\ If $\tau_1/\beta < |\arg x|\le \pi$\ and $|\arg y|\le \tau_1/\alpha$, then
\begin{flalign} \label{as3}
E_{\alpha,\beta}(x,y;\mu) &= \frac{1}{\beta}\, \frac{ \re^{y^{1/\beta}} %
y^{\frac{1 + \alpha - \mu}{\beta}} }{y^{\alpha/\beta} - x}\nonumber\\
+& \lsum_{n=1}^{p_\beta} \lsum_{m=1}^{p_\alpha} \frac{x^{-n} y^{-m}}{\Gamma(\mu - n\alpha - \m\beta)} %
+ o\l(|xy|^{-1} |x|^{-p_\beta}\r) + o\l(|xy|^{-1} |y|^{-p_\alpha}\r);
\end{flalign}

\item[{\rm 4)}]\ If $\tau_1/\beta < |\arg x|\le \pi$\ and $\tau_1/\alpha <|\arg y|\le \pi$, then
\begin{flalign} \label{as4}
E_{\alpha,\beta}(x,y;\mu) =& \lsum_{n=1}^{p_\beta} \lsum_{m=1}^{p_\alpha} \frac{x^{-n} y^{-m}} %
{\Gamma(\mu - n\alpha - \m\beta)}\nonumber\\
&+ o\l(|xy|^{-1} |x|^{-p_\beta}\r) + o\l(|xy|^{-1} |y|^{-p_\alpha}\r).
\end{flalign}
\ei
\end{thm}
\begin{proof}
The proof below focuses on the first case since, in the three other cases, the proofs are easily completed along the same lines as the one in case 1), that is the proof of asymptotic formula~\eqref{as1}.

So, under the constraints in case 1) (i.e., $|\arg x|\le \tau_1/\beta$\ and $|\arg y|\le \tau_1/\alpha$), pick a real number $\tau_2$ satisfying the inequalities~\eqref{cond2}
\begin{equation} \label{cond2}
\pi \alpha\beta/2 < \tau_1 < \tau_2\le \min\bl(\pi, \pi \alpha\beta\br).
\end{equation}
It is easy to show the expansion
\begin{equation} \label{exp}
\frac{1}{(\z^{1/\beta} - x)(\z^{1/\alpha} - y)} = \lsum_{n=1}^{p_\beta} \lsum_{m=1}^{p_\alpha} %
\frac{ \z^{\frac{n - 1}{\alpha} + \frac{m - 1}{\beta}} }{x^n y^m} + %
\frac{ x^{p_\beta} \z^{ \frac{p_\alpha}{\alpha} } + y^{p_\alpha} \z^{ \frac{p_\beta}{\beta} } %
- \z^{ \frac{p_\alpha}{\alpha} + \frac{p_\beta}{\beta}} }{x^{p_\beta} y^{p_\alpha} %
(\z^{1/\beta} - x)(\z^{1/\alpha} - y) }\,.
\end{equation}
Here we use the formula~\eqref{mlrep4} from Lemma~\ref{lem3}. Set $\eps = 1$ in~\eqref{exp}, then to the right of the contour $\gamma(1;\tau_2)$ (that is, within the complex region $\Omega^{(+)}(1;\tau_2)$), in view of expansion~\eqref{exp} and by Eq.~\eqref{mlrep4} the integral representation of $E_{\alpha,\beta}(x,y ;\mu)$ can be expressed in the form

\begin{flalign} \label{mlrep5}
E_{\alpha,\beta}(x,y;\mu) &= \frac{1}{\alpha}\, \frac{ \re^{x^{1/\alpha}} %
x^{\frac{1 + \beta - \mu}{\alpha}} }{x^{\beta/\alpha} - y} + \frac{1}{\beta}\, \frac{\re^{y^{1/\beta}} %
y^{\frac{1 + \alpha - \mu}{\beta}} } {y^{\alpha/\beta} - x}\nonumber\\
&+ \lsum_{n=1}^{p_\beta} \lsum_{m=1}^{p_\alpha} \frac{1}{2\pi\ri} \frac{1}{\alpha\beta}\, %
\l(\lint_{\gamma(1;\tau_2)} \re^{\z^{1/(\alpha \beta)}} \z^{ \frac{\alpha + \beta - \mu}{\alpha \beta} %
- 1 + \frac{n - 1}{\beta} + \frac{m - 1}{\alpha} }\, \rd \z)\r)\, x^{-n} y^{-m}\nonumber\\
&+ \frac{1}{2\pi\ri} \frac{1}{\alpha \beta} \lint_{\gamma(1;\tau_2)} \re^{ \z^{1/(\alpha \beta)} } %
\z^{ \frac{\alpha + \beta - \mu}{\alpha \beta} - 1 }\; \frac{ x^{p_\beta} \z^{\frac{p_\alpha}{\alpha}} %
+ y^{p_\alpha} \z^{\frac{p_\beta}{\beta}} - \z^{\frac{p_\alpha}{\alpha} + \frac{p_\beta}{\beta}} } %
{x^{p_\beta} y^{p_\alpha} (\z^{1/\beta} - x )(\z^{1/\alpha} - y) }\, \rd \z.
\end{flalign}
By Hankel's formula, the integral representation of the reciprocal gamma function~(see e.g. \cite[\color{cyan}Eq.~C3]{Lavault17}, \cite[\color{cyan}Chap.~3, \S3.2.6]{Temme96}, etc.) writes
\begin{flalign*}
& \kern3cm \frac{1}{\Gamma(s)} = \lint_{\gamma(\eps;\tau)} \re^u u^{-s}\, \rd u,\\
\intertext{%
and, as a consequence, the summand of the second term in~\eqref{mlrep5} satisfies the identity
}
\frac{1}{2\pi\ri} \frac{1}{\alpha\beta}\, & \lint_{\gamma(1;\tau_2)} \re^{\z^{1/(\alpha \beta)}} %
\z^{ \frac{\alpha + \beta - \mu}{\alpha \beta} - 1 + \frac{n - 1}{\beta} + \frac{m - 1}{\alpha} }\, %
\rd \z = \frac{1}{2\pi\ri} \frac{1}{\alpha\beta}\, \lint_{\gamma(1;\tau_2)} \re^{\z^{1/(\alpha \beta)}} %
\z^{ \frac{1- \mu}{\alpha \beta} - 1 + \frac{n}{\beta} + \frac{m}{\alpha} }\, \rd \z\nonumber\\
&= \frac{1}{2\pi\ri} \frac{1}{\alpha\beta}\, \lint_{\gamma(1;\tau_2)} \re^{\z^{1/(\alpha \beta)}} %
\z^{ -\frac{1}{\alpha \beta}\bl(\mu - \alpha n - \beta m\br) + \frac{1}{\alpha \beta} - 1 }\, \rd \z %
= \frac{1}{\Gamma(\mu - \alpha n - \beta m)}\,.
\end{flalign*} 
Therefore, under the constraints resulting from inequalities~\eqref{cond2}, Eq.~\eqref{mlrep5} can be transformed into
\begin{flalign} \label{mlrep6}
E_{\alpha,\beta}(x,y;\mu) &= \frac{1}{\alpha}\, \frac{ \re^{x^{1/\alpha}} %
x^{\frac{1 + \beta - \mu}{\alpha}} }{x^{\beta/\alpha} - y} + \frac{1}{\beta}\, \frac{\re^{y^{1/\beta}} %
y^{\frac{1 + \alpha - \mu}{\beta}} } {y^{\alpha/\beta} - x} + \lsum_{n=1}^{p_\beta} %
\lsum_{m=1}^{p_\alpha} \frac{x^{-n} y^{-m}}{\Gamma(\mu - \alpha n - \beta m)}\nonumber\\
&+ \frac{1}{2\pi\ri} \frac{1}{\alpha \beta} \lint_{\gamma(1;\tau_2)} \re^{ \z^{1/(\alpha \beta)} } %
\z^{ \frac{1 + \alpha + \beta - \mu}{\alpha \beta} - 1 }\; \frac{ x^{p_\beta} \z^{p_\alpha/\alpha} %
+ y^{p_\alpha} \z^{p_\beta/\beta} - \z^{ p_\alpha/\alpha + p_\beta/\beta } } %
{x^{p_\beta} y^{p_\alpha} (\z^{1/\beta} - x)(\z^{1/\alpha} - y) }\, \rd \z.
\end{flalign}
Next, simplifying the last term in Eq.~\eqref{mlrep6} yields
\begin{flalign} \label{simpmlrep6}
\frac{1}{2\pi\ri} \frac{1}{\alpha \beta}\, \lint_{\gamma(1;\tau_2)} \re^{ \z^{1/(\alpha \beta)} } %
& \z^{ \frac{1 + \alpha + \beta - \mu}{\alpha \beta} - 1 }\; \frac{ x^{p_\beta} \z^{p_\alpha/\alpha} %
+ y^{p_\alpha} \z^{p_\beta/\beta} - \z^{p_\alpha/\alpha + p_\beta/\beta} } %
{x^{p_\beta} y^{p_\alpha} (\z^{1/\beta} - x)(\z^{1/\alpha} - y) }\, \rd \z\nonumber\\
=& \frac{1}{2\pi\ri} \frac{1}{\alpha \beta}\, \lint_{\gamma(1;\tau_2)}\, %
\frac{ \re^{ \z^{1/(\alpha \beta)} } \z^{ \frac{ 1 + \alpha + \beta - \mu}{\alpha \beta} - 1 %
+ p_\alpha/\alpha } } {y^{p_\alpha} (\z^{1/\beta} - x)(\z^{1/\alpha} - y)}\, \rd \z\nonumber\\
+& \frac{1}{2\pi\ri}  \frac{1}{\alpha \beta} \lint_{\gamma(1;\tau_2)} %
\frac{ \re^{ \z^{1/(\alpha \beta)} } \z^{ \frac{ 1 + \alpha + \beta - \mu}{\alpha \beta} - 1 %
+ p_\beta/\beta } } {x^{p_beta} (\z^{1/\beta} - x)(\z^{1/\alpha} - y)}\, \rd \z\nonumber\\
-& \frac{1}{2\pi\ri} \frac{1}{\alpha \beta}\, \lint_{\gamma(1;\tau_2)}\, %
\frac{ \re^{ \z^{1/(\alpha \beta)} } \z^{ \frac{ 1 + \alpha + \beta - \mu}{\alpha \beta} - 1 + %
p_\alpha/\alpha + p_\beta/\beta } } {x^{p_\beta} y^{p_\alpha}(\z^{1/\beta} - x)(\z^{1/\alpha} - y)}\, %
\rd \z = I_1 + I_2 + I_3.
\end{flalign}
Assuming $|\arg x|\le \tau_1/\beta$\ and $|\arg y|\le \tau_1/\alpha$, each integral $I_1$, $I_2$ and $I_3$ in~\eqref{simpmlrep6} can be evaluated for large values of $|x|$\ and $|y|$, and provided that 
$|\arg x|\le \tau_1/\beta$ and $|x|$ is large enough, it can be checked that
\[
\min_{\z\in \gamma(1;\tau_2)} \l|\z^{1/\beta} - x\r| = |x| \sin(\tau_2/\beta - \tau_1/\beta) %
= |x| \sin\l(\tfrac{\tau_2 - \tau_1}{\beta}\r).\]
Similarly, when $|\arg y|\le \tau_1/\beta$ and $|y|$ is large enough,
\[
\min_{\z\in \gamma(1;\tau_2)} \l|\z^{1/\beta} - y\r| = |y| \sin(\tau_2/\alpha - \tau_1/\alpha) %
= |y| \sin\l(\tfrac{\tau_2 - \tau_1}{\alpha}\r).\]

Hence, for large $|x|$\ and $|y|$ with $|\arg x|\le \tau_1/\beta$\ and $|\arg y|\le \tau_1/\alpha$, we obtain an estimate of the integral $I_1$
\begin{equation} \label{estimateI1}
|I_1|\le \frac{ |x|^{-1} |y|^{-p_\alpha -1} }{ 2\pi \alpha \beta %
\sin\l(\frac{\tau_2 - \tau_1}{\alpha}\r) \sin\l(\frac{\tau_2 - \tau_1}{\beta}\r) }\, %
\lint_{\gamma(1;\tau_2)} \l|\re^{ \z^{\frac{1}{\alpha \beta}} }\r| %
\l|\z^{ \frac{ 1 + \alpha + \beta - \mu}{\alpha \beta} - 1 + \frac{p_\alpha}{\alpha} }\r|\, |\rd \z|.
\end{equation}
Besides, since the rays defined by $S_{\tau_2} = \bl\{\arg \z = \pm \tau_2,\ |\z|\le 1\br\}$ belong to the contour $\gamma(1;\tau_2)$, the integral in inequality~\eqref{estimateI1} is convergent. Wherefrom we get the equality
\[
\l|\re^{ \z^{1/(\alpha \beta)} }\r| = \exp\l(|\z|^{\frac{1}{\alpha \beta}} %
\cos \tfrac{\tau_2}{\alpha \beta} \r).\]
Now, according to inequalities~\eqref{cond2}, we have $\cos \frac{\tau_2}{\alpha \beta} < 0$. Thus,
\[
I_1 = o\l(|xy|^{-1} |y|^{-p_\alpha}\r)\ \qquad \tand \qquad I_2 = o\l(|xy|^{-1} |x|^{-p_\beta}\r).\]

Furthermore, referring to the last term in Eq.~\eqref{simpmlrep6} yields the following estimate of $I_3$,
\begin{equation} \label{estimateI3}
|I_3|\le \frac{ |x|^{-p_\beta - 1} |y|^{-p_\alpha -1} }{ 2\pi \alpha\beta %
\sin\l(\frac{\tau_2 - \tau_1}{\alpha}\r) \sin\l(\tfrac{\tau_2 - \tau_1}{\beta}\r) }\, %
\lint_{\gamma(1;\tau_2)} \l|\re^{ \z^{1/(\alpha \beta)} }\r| \l|\z^{ \frac{1 + \alpha + \beta - \mu} %
{\alpha \beta} - 1 + p_\alpha/\alpha + p_\beta/\beta }\r|\, |\rd \z|,
\end{equation}
which gives $I_3 = o\l(|xy|^{-1} |x|^{-p_\beta} |y|^{-p_\alpha}\r)$.

Hence, this leads finally to the overall asymptotic formula 
\[
I_1 + I_2 + I_3 = o\l(|xy|^{-1} |x|^{-p_\beta}\r) + o\l(|xy|^{-1} |y|^{-p_\alpha}\r)\]
and the proof of Eq.~\eqref{as1} (in case 1 of Theorem~\ref{as}) follows. 

Similarly, the proofs of Eqs.~\eqref{as2} (case 2), Eqs.~\eqref{as3} (case 3) and Eqs.~\eqref{as4} (case 4) run along the same lines as the above proof of Eq.~\eqref{as1} (case 1), and the proof of Theorem~\ref{as} is completed.
\end{proof}

\newpage
\addcontentsline{toc}{section}{References}
\section*{References}

\vskip -1cm
\def\refname{\empty}

\bibliographystyle{article}
\def\bibfmta#1#2#3#4{ {\sc #1}, {#2}, \emph{#3}, #4.}
\bibliographystyle{book}
\def\bibfmtb#1#2#3#4{ {\sc #1}, \emph{#2}, {#3}, #4.}

\vskip 1cm
\no\rule{\textwidth}{.5pt}
\vskip .3cm
\no {\small Christian {\sc Lavault}
\newline \emph{E-mail:} \href{mailto:lavault@lipn.univ-paris13.fr}{lavault@lipn.univ-paris13.fr}
\newline LIPN, UMR CNRS 7030 -- Laboratoire d'Informatique de Paris-Nord
\newline Universit\'e Paris 13, Sorbonne Paris Cit\'e, F-93430 Villetaneuse.
\newline \emph{URL:} \url{http://lipn.univ-paris13.fr/\textasciitilde lavault}
}

\end{document}